\documentclass[10pt]{article}
\usepackage{graphicx}
\usepackage{makeidx}

\usepackage{amsmath}
\usepackage{amsthm}
\usepackage{amsfonts}
\usepackage{amssymb}
\usepackage{epsfig}
\usepackage{tabularx}
\usepackage{latexsym}
\newtheorem{lemma}{Lemma}

\def\qed{$\Box$}

\date{}

\begin{document}

\title{Graph-theoretic perspective on a special class of Steiner Systems}
\author{
Jithin Mathews 
\footnote{Alumnus, IIT Guwahati, Assam-781039,
India. Email: m.jithin@iitg.ernet.in} 
}
\maketitle


\begin{abstract}
We study $S(t-1,t,2t)$, which is a special class of Steiner systems. Explicit constructions for designing such systems are developed under a graph-theoretic platform where Steiner systems are represented in the form of uniform hypergraphs. The constructions devised are then used to study the $2$-coloring properties of these uniform hypergraphs.
\vspace{1mm}

\noindent \textbf{Keywords:} Steiner System; Uniform Hypergraph; Hypergraph Coloring
\end{abstract}

\section{Introduction}
\label{introduction}

A Steiner system, commonly denoted as $S(t,k,n)$, is a combinatorial design with a set $V$ of $n$ points, and a collection of subsets of $V$ of size $k$ (called blocks), such that any $t$-subset of $V$ is contained in exactly one of the blocks. The study of Steiner systems and their generalizations is closely related to several areas of mathematics, such as geometry, group theory and coding theory. In addition, Steiner systems have important applications in computer science and cryptography \cite{col2}. A lot of effort has been spend on designing particular Steiner systems. For a survey on the known Steiner systems, the reader is referred to \cite{col1}.

Constructing general Steiner systems is one of the most important problems in design theory. A finite projective plane $S(2,q+1,q^2+q+1)$ and a finite affine plane $S(2, q, q^2)$ \cite{col1} are examples for such systems. Here we study on a special class of Steiner systems, $viz.$, $S(t-1,t,2t)$. Only two such non-trivial ($t>2$) systems are known to exist: $S(3,4,8)$, and $S(5,6,12)$ which is the famous Mathieu group $M_{12}$ \cite{math}. An alternative construction for $M_{12}$ can be obtained by the use of the `kitten' of R.T. Curtis \cite{curt}. These systems can be used to construct $S(2,3,7)$ and $S(4,5,11)$, respectively. It becomes possible due to the fact that if a Steiner system $S(t,k,n)$ exists, then taking all blocks containing a specific element and discarding that element from them produces a derived Steiner system $S(t-1,k-1,n-1)$. Therefore, the existence of $S(t-1,k-1,n-1)$ is a necessary condition for the existence of $S(t,k,n)$.

In this note, we give explicit constructions for designing the two known Steiner systems that comes under the previously mentioned class of general systems: $S(t-1,t,2t)$. For the same, we follow a graph-theoretic approach where Steiner systems are represented in the form of $t$ uniform hypergraphs. These constructions are then used to study the $2$-coloring properties of the hypergraphs designed. 

\section{Constructing $S(t-1,t,2t)$}
We define the Steiner system $S(t-1,t,2t)$ as $t$-uniform hypergraphs, $ie.,$ hypergraphs where every hyperedge contains exactly $t$ vertices. We denote these hypergraphs as $S_{t}^{2t}$ and construct them  for $t=6$ and $t=4$. The vertex set and the hyperedge set of $S_{t}^{2t}$ is denoted by $V_{t}^{2t}$ and $E_{t}^{2t}$, respectively, where $V_{t}^{2t}$ = $\{v_1, v_2, v_3, \ldots, v_{2t-1}, v_{2t}\}$.

The following three steps describes a high-level overview of the significant steps in the design of $S(t-1,t,2t)$:

\noindent$(i)$ \textit{Existence of complementary hyperedges:} 
The complementary hyperedge $h'$ corresponding to a particular $t$-uniform hyperedge $h$ is defined by $V_{t}^{2t} \setminus h$. The existence of these complementary hyperedges is a mandatory condition in this design. For example, the hyperedge $\{v_1, v_2, v_3, v_4, v_5, v_6\}$ will be present in the hypergraph $S_{6}^{12}$ if and only if its complementary hyperedge $\{v_7, v_8, v_9, v_{10}, v_{11}, v_{12}\}$ is also present in it. We begin the construction of $S_{t}^{2t}$ by adding a $t$-uniform hyperedge $h$ into the hyperedge set $E_{t}^{2t}$ along with its complementary hyperedge $h'$. Note that, $h$ which is the first hyperedge added, can be any $t$-subset of $V_{t}^{2t}$.

\noindent$(ii)$ \textit{The non-trivial hyperedge formation:} Here we design all the hyperedges that intersect  exactly on $(t-2)$ vertices of $h$ containing a particular vertex $v$, $v \in h$. Here we define the non-trivial steps of our construction. There are two non-trivial steps in the design of the hypergraph $S_{6}^{12}$ which represents the Steiner system $S(5,6,12)$, while only a single such step is present in the design of the hypergraph $S_{4}^{8}$.

\noindent$(iii)$ \textit{The trivial hyperedge formation:} In this step, hyperedges are formed naturally using the definition of the Steiner system $S(t-1,t,2t)$ that every $(t-1)$ subset of $V$ must occur exactly once.

Having defined an overview for our design, let us now devise the explicit constructions for both $S(5,6,12)$ and $S(3,4,8)$.

\subsection{Constructing $S_{t}^{2t}$ when $t=6$}
As mentioned before, we call the hypergraph corresponding to the Steiner system $S(5,6,12)$ as $S_{6}^{12}$. We begin with an empty hyperedge-set of $S_{6}^{12}$, denoted by $E_{6}^{12}$, and its vertex-set is denoted by $V_{6}^{12}$ = $\{v_1, v_2, v_3, \ldots, v_{11},$ $v_{12}\}$. Now let us start devising the hyperedges for the hypergraph $S_{6}^{12}$:\\

\noindent $\textbf{Step (1)}$ Let the hyperedge $\{v_1, v_2, v_3, v_4, v_5, v_6\}$ be denoted by $h$ and its
complementary hyperedge be $h'$, and $h'$ = $\{v_7, v_8, v_9, v_{10}, v_{11}, v_{12}\}$.
We start the construction by adding the hyperedges $h$ and $h'$ into the hyperedge-set $E_{6}^{12}$.\\

\noindent $\textbf{Step (2)}$ Here we devise and add to $E_{6}^{12}$, all the hyperedges that intersect  exactly on $4$ vertices of $h$ containing a particular vertex, and $w.l.o.g.$ we choose $v_1$ to be that vertex. Note that for every hyperedge added to $E_{6}^{12}$, we also add its corresponding complementary hyperedge. The following two sub-steps defines certain sets that are useful in our construction:\\

$\textbf{Step (2.1)}$ If set $A$ = $\{1', 2', 3', x\}$ and set $B$ = $\{\{a, b\}, \{c, d\}, \{e, f\} \}$, then,\\

$\hspace{18mm}$ for set $A'$ = $\{1', 2', 3', y\}$, set $B'$ = $\{\{a, c\}, \{b, e\}, \{d, f\} \}$, and,

$\hspace{18mm}$ for set $A''$ = $\{1', 2', 3', z\}$, set $B''$ = $\{\{a, f\}, \{b, d\}, \{c, e\} \}$.\\ 

$\textbf{Step (2.2)}$ If set $A$ = $\{1', 2', 3', x\}$ and set $B$ = $\{\{a, b\}, \{c, d\}, \{e, f\} \}$, and,

$\hspace{18mm}$ if set $A'$ = $\{1', 2', 3', y\}$ and set $B'$ = $\{\{a, c\}, \{b, e\}, \{d, f\} \}$, then,\\

$\hspace{18mm}$ for set $A''$ = $\{1', 2', x, y\}$, set $B''$ = $\{\{a, e\}, \{b, d\}, \{c, f\} \}$, and,

$\hspace{18mm}$ for set $A'''$ = $\{1', 3', x, y\}$, set $B'''$ = $\{\{a, f\}, \{b, c\}, \{d, e\} \}$.\\

In order to apply these steps, we need to map the elements of the above sets to the elements of $V_{6}^{12}$. Note the following vertex-mapping:\\
$(1'\rightarrow v_1)$, $(2'\rightarrow v_2)$, $(3'\rightarrow v_3)$, $(x\rightarrow v_4)$, $(y\rightarrow v_5)$, $(z\rightarrow v_6)$, $(a\rightarrow v_7)$, $(b\rightarrow v_8)$, $(c\rightarrow v_9)$, $(d\rightarrow v_{10})$, $(e\rightarrow v_{11})$ and $(f\rightarrow v_{12})$.\\\\
Now, apply Step $(2.1)$ on these mapped elements to get the following sets:\\
$A$ = $\{v_1, v_2, v_3, v_4\}$ and $B$ = $\{\{v_7, v_8\}, \{v_9, v_{10}\}, \{v_{11}, v_{12}\} \}$,\\
$A'$ = $\{v_1, v_2, v_3, v_5\}$ and $B'$ = $\{\{v_7, v_9\}, \{v_8, v_{11}\}, \{v_{10}, v_{12}\} \}$, 
and,\\
$A''$ = $\{v_1, v_2, v_3, v_6\}$ and $B''$ = $\{\{v_7, v_{12}\}, \{v_8, v_{10}\}, \{v_9, v_{11}\} \}$.\\

Perform the Cartesian product $A \times B$, and add the three $6$-uniform hyperedges formed, along with their complementary hyperedges, into the hyperedge-set $E_{6}^{12}$. Similarly, perform $A' \times B'$ and $A'' \times B''$ and add the hyperedges thus formed (along with their corresponding complementary hyperedges) to $E_{6}^{12}$. Note that here we have added $18$ hyperedges to $E_{6}^{12}$.

In order to form the next few hyperedges we have to apply the Step $(2.2)$. We use the previous vertex-mapping for the same. After applying Step $(2.2)$ we get the following sets:\\\\
$A$ = $\{v_1, v_2, v_3, v_4\}$ and $B$ = $\{\{v_7, v_8\}, \{v_9, v_{10}\}, \{v_{11}, v_{12}\} \}$,\\
$A'$ = $\{v_1, v_2, v_3, v_5\}$ and $B'$ = $\{\{v_7, v_9\}, \{v_8, v_{11}\}, \{v_{10}, v_{12}\} \}$,\\
$A''$ = $\{v_1, v_2, v_4, v_5\}$ and $B''$ = $\{\{v_7, v_{11}\}, \{v_8, v_{10}\}, \{v_9, v_{12}\} \}$, and,\\
$A'''$ = $\{v_1, v_3, v_4, v_5\}$ and $B'''$ = $\{\{v_7, v_{12}\}, \{v_8, v_9\}, \{v_{10}, v_{11}\} \}$.

Perform $A'' \times B''$ and $A''' \times B'''$, and add the hyperedges thus formed (along with their corresponding complementary hyperedges) to $E_{6}^{12}$. Note that here we have added $12$ more hyperedges to $E_{6}^{12}$.
 
For creating all the remaining hyperedges that intersect exactly on $4$ vertices of $h$ containing the vertex $v_1$ we need to use Step $(2.1)$. For the same, we choose two sets from the previously created sets (sets with their elements as the vertices of $S_{6}^{12}$) $A$, $A'$, $A''$ and $A'''$, along with their corresponding $B$, $B'$, $B''$ and $B'''$. These sets are then mapped to the sets $A$ and $A'$ and to their corresponding sets $B$ and $B'$ which are defined in the Step $(2.1)$ in-order to create a new mapping for the vertices of $V_{6}^{12}$. This helps us to get a new $A''$ and a new $B''$ (of Step $(2.1)$) that can aid us in the formation of more hyperedges.
For example, mapping the previously defined sets $A$, $B$ and $A'''$, $B'''$ to the sets $A$, $B$ and $A'$, $B'$, respectively, of the Step $(2.1)$, we get the following mapping to the vertices of $V_{6}^{12}$:\\\\
$(1'\rightarrow v_1)$, $(2'\rightarrow v_2)$, $(3'\rightarrow v_4)$, $(x\rightarrow v_3)$, $(y\rightarrow v_5)$, $(a\rightarrow v_7)$, $(b\rightarrow v_8)$, $(c\rightarrow v_{11})$, $(d\rightarrow v_{12})$, $(e\rightarrow v_{10})$, $(f\rightarrow v_9)$ and $(z\rightarrow v_6)$.\\

\noindent The following sets are formed on applying the above mapping on Step $(2.1)$:\\\\
$A$ = $\{v_1, v_2, v_4, v_3\}$ and $B$ = $\{\{v_7, v_8\}, \{v_{11}, v_{12}\}, \{v_{10}, v_9\} \}$,\\
$A'$ = $\{v_1, v_2, v_4, v_5\}$ and $B'$ = $\{\{v_7, v_{11}\}, \{v_8, v_{10}\}, \{v_{12}, v_9\} \}$, and,\\
$A''$ = $\{v_1, v_2, v_4, v_6\}$ and $B''$ = $\{\{v_7, v_9\}, \{v_8, v_{12}\}, \{v_{11}, v_{10}\} \}$.

Perform $A'' \times B''$ and add the hyperedges formed (along with their corresponding complementary hyperedges) to $E_{6}^{12}$. Here we have added $6$ more hyperedges to $E_{6}^{12}$. 

Using the same method one can create more vertex-mappings for the vertices of $V_{6}^{12},$ and each such mapping will give us a new version of the sets $A''$ and $B''$. Let us see another example: mapping the previously defined sets $A'$, $B'$ and $A'''$, $B'''$ to the sets $A$, $B$ and $A'$, $B'$, respectively, of the Step $(2.1)$, we get the following mapping to the vertices of $V_{6}^{12}$:\\\\
$(1'\rightarrow v_1)$, $(2'\rightarrow v_2)$, $(3'\rightarrow v_5)$, $(x\rightarrow v_3)$, $(y\rightarrow v_4)$, $(a\rightarrow v_7)$, $(b\rightarrow v_9)$, $(c\rightarrow v_{11})$, $(d\rightarrow v_8)$, $(e\rightarrow v_{12})$, $(f\rightarrow v_{10})$ and $(z\rightarrow v_6)$.\\

\noindent The following version of sets $A''$ and $B''$ are formed on applying the above vertex-mapping on Step $(2.1)$:\\\\
$A''$ = $\{v_1, v_2, v_5, v_6\}$ and $B''$ = $\{\{v_7, v_{10}\}, \{v_9, v_8\}, \{v_{11}, v_{12}\} \}$.

As in the previous case, perform $A'' \times B''$ and add the hyperedges formed (along with their corresponding complementary hyperedges) to $E_{6}^{12}$. Note that here we have added $6$ more hyperedges to $E_{6}^{12}$.

Following the same technique, one can create more such vertex-mappings that forms different versions for sets $A''$ and $B''$ on applying the Step $(2.1)$, and the Cartesian product of these different versions of $A''$ and $B''$ will produce hyperedges that are useful for our design. Three more such versions of $A''$ and $B''$ are created in this step of our construction:\\

\noindent $A''$ = $\{v_1, v_3, v_4, v_6\}$ and $B''$ = $\{\{v_7, v_{10}\}, \{v_8, v_{11}\}, \{v_9, v_{12}\} \}$.\\
$A''$ = $\{v_1, v_3, v_5, v_6\}$ and $B''$ = $\{\{v_7, v_{11}\}, \{v_9, v_{10}\}, \{v_8, v_{12}\} \}$.\\
$A''$ = $\{v_1, v_4, v_5, v_6\}$ and $B''$ = $\{\{v_7, v_8\}, \{v_{11}, v_9\}, \{v_{10}, v_{12}\} \}$.\\
 
Performing the Cartesian product of each $A''$ with its corresponding $B''$ will give us more $6$-uniform hyperedges. Add them into $E_{6}^{12}$ along with their corresponding complementary hyperedges, and we have added a total of $18$ hyperedges here.

This completes the Step $(2)$ of constructing $S_{6}^{12}$, where we have added all the hyperedges that intersect exactly on $4$ vertices of $h$ containing the vertex $v_1$ along with their corresponding complementary hyperedges. In this step, a total of $60$ hyperedges are added into the hyperedge-set $E_{6}^{12}$.\\

\noindent $\textbf{Step (3)}$
Here we devise and add to $E_{6}^{12}$, all the hyperedges that intersect  exactly on $3$ vertices of $h$, or, exactly on $2$ vertices of $h$, that contains the vertex $v_1$. For every hyperedge added to $E_{6}^{12}$, we also add its corresponding complementary hyperedge. In this step, hyperedges can be formed quite naturally using the fundamental property of a Steiner system $S(t-1,t,2t)$ that every $(t-1)$ subset of $V$ must occur exactly once.\\

$\textbf{(a)}$ We take all the triples (or $3$-subsets) from the set $h$ that contains the vertex $v_1$ and add to it $3$ vertices from the set $h'$ in-order to form $6$-uniform hyperedges such that all the $5$-subsets formed from the addition of any $3$ vertices of $h$ (containing $v_1$) with $2$ vertices of $h'$ occur exactly once. Consider a $3$-subset $\{v_1, b, c\}$ of $h$. One can easily verify that in Step $(2)$, during the formation of hyperedges, the triple $\{v_1, b, c\}$ is connected with $9$ different $2$-subsets (or pairs) of $h'$ to form $9$ different $5$-subsets exactly once. The total number of pairs in $h'$ is equal to ${6 \choose 2} = 15$. Therefore, we need to construct new hyperedges such that the remaining six $5$-subsets containing the triple $\{v_1, b, c\}$ and a pair from the set $h'$ are formed exactly once. In fact, it is possible to satisfy this condition.

For example, consider the triple $\{v_1, v_2, v_3\}$ of $h$. The $6$ remaining pairs of $h'$ that are to be connected with $\{v_1, v_2, v_3\}$ to form hyperedges are $\{v_7, v_{10}\}$, $\{v_{10}, v_{11}\}$, $\{v_7, v_{11}\}$, $\{v_8, v_9\}$, $\{v_8, v_{12}\}$ and $\{v_9, v_{12}\}$. Note that, in Step $(2)$ all the remaining pairs of $h'$ have been added to the triple $\{v_1, v_2, v_3\}$ for the formation of
hyperedges. Hence, we have $A$ = $\{v_1, v_2, v_3\}$ and $B$ = $\{\{v_7, v_{10}, v_{11}\}, \{v_8, v_9, v_{12}\} \}$. Perform $A\times B$ and add the $2$ hyperedges formed (along with their corresponding complementary hyperedges) into $E_{6}^{12}$. In the same way, one can devise more hyperedges using the remaining $9$ triples of $h$ that contains the vertex $v_1$. The following sets defines the $9$ different versions of set $A$ and set $B$ which can be used to construct all the remaining hyperedges that intersect exactly on $3$ vertices of $h$ containing the vertex $v_1$:\\

\noindent $A$ = $\{v_1, v_2, v_4\}$ and $B$ = $\{\{v_7, v_{10}, v_{12}\}, \{v_8, v_9, v_{11}\} \}$\\
$A$ = $\{v_1, v_2, v_5\}$ and $B$ = $\{\{v_7, v_8, v_{12}\}, \{v_9, v_{10}, v_{11}\} \}$\\
$A$ = $\{v_1, v_2, v_6\}$ and $B$ = $\{\{v_7, v_8, v_{11}\}, \{v_9, v_{10}, v_{12}\} \}$\\

\noindent $A$ = $\{v_1, v_3, v_4\}$ and $B$ = $\{\{v_7, v_9, v_{11}\}, \{v_8, v_{10}, v_{12}\} \}$\\
$A$ = $\{v_1, v_3, v_5\}$ and $B$ = $\{\{v_7, v_8, v_{10}\}, \{v_9, v_{11}, v_{12}\} \}$\\
$A$ = $\{v_1, v_3, v_6\}$ and $B$ = $\{\{v_7, v_8, v_9\}, \{v_{10}, v_{11}, v_{12}\} \}$\\

\noindent $A$ = $\{v_1, v_4, v_5\}$ and $B$ = $\{\{v_7, v_9, v_{10}\}, \{v_8, v_{11}, v_{12}\} \}$\\
$A$ = $\{v_1, v_4, v_6\}$ and $B$ = $\{\{v_7, v_{11}, v_{12}\}, \{v_8, v_9, v_{10}\} \}$\\
$A$ = $\{v_1, v_5, v_6\}$ and $B$ = $\{\{v_7, v_9, v_{12}\}, \{v_8, v_{10}, v_{11}\} \}$

Perform the Cartesian product of the corresponding sets $A$ and $B$, and add the hyperedges formed (along with their corresponding complementary hyperedges) into the hyperedge-set $E_{6}^{12}$. Note that, we have constructed a total of $40$ hyperedges in this step of our construction.\\

$\textbf{(b)}$ Here we take all the pairs ($2$-subsets) from the set $h$ that contains the vertex $v_1$ and add to it $4$ vertices from the set $h'$ in-order to form $6$-uniform hyperedges such that all the $5$-subsets formed from the addition of any $2$ vertices of $h$ (containing $v_1$) with $3$ vertices of $h'$ occur exactly once. Consider a $2$-subset $\{v_1, b\}$ of $h$. One can easily verify that in the last step (Step $3(a)$), during the formation of hyperedges, the pair $\{v_1, b\}$ is connected with $8$ different $3$-subsets (or triples) of $h'$ to form $8$ different $5$-subsets exactly once. The total number of triples in $h'$ is equal to ${6 \choose 3} = 20$. Therefore, we need to construct new hyperedges such that the remaining twelve $5$-subsets containing the pair $\{v_1, b\}$ and a triple from the set $h'$ are formed exactly once. 

For example, consider the pair $\{v_1, v_2\}$ of $h$. The $12$ remaining triples of $h'$ that are to be connected with $\{v_1, v_2\}$ to form hyperedges are $\{v_7, v_8, v_9\}$, $\{v_7, v_8, v_{10}\}$, $\{v_7, v_9, v_{10}\}$, $\{v_8, v_9, v_{10}\}$, $\{v_7, v_9, v_{11}\}$, $\{v_7, v_9, v_{12}\}$, $\{v_7, v_{11}, v_{12}\}$, $\{v_9, v_{11}, v_{12}\}$, $\{v_8, v_{10}, v_{11}\}$, $\{v_8, v_{10}, v_{12}\}$, $\{v_8, v_{11}, v_{12}\}$ and $\{v_{10}, v_{11}, v_{12}\}$. Note that, in Step $3(a)$ all the remaining triples of $h'$ have been added to the pair $\{v_1, v_2\}$ for the formation of hyperedges. Hence, we have $A$ = $\{v_1, v_2\}$ and $B$ = $\{\{v_7, v_8, v_9, v_{10}\}, \{v_7,$ $v_9, v_{11}, v_{12}\}, \{v_8, v_{10}, v_{11}, v_{12}\}\}$. Perform $A\times B$ and add the $3$ hyperedges formed (along with their corresponding complementary hyperedges) into $E_{6}^{12}$. In the same way, one can devise more hyperedges using the remaining $4$ pairs of $h$ that contains $v_1$. The following sets defines the $4$ different versions of set $A$ and set $B$ which can be used to construct all the remaining hyperedges that intersect exactly on $2$ vertices of $h$ containing the vertex $v_1$:\\

\noindent
$A$ = $\{v_1, v_3\}$ and $B$ = $\{\{v_7, v_8, v_{11}, v_{12}\}, \{v_8, v_9, v_{10}, v_{11}\}, \{v_7, v_9, v_{10}, v_{12}\}\}$\\
$A$ = $\{v_1, v_4\}$ and $B$ = $\{\{v_7, v_8, v_9, v_{12}\}, \{v_7, v_8, v_{10}, v_{11}\}, \{v_9, v_{10}, v_{11}, v_{12}\}\}$\\
$A$ = $\{v_1, v_5\}$ and $B$ = $\{\{v_7, v_8, v_9, v_{11}\}, \{v_8, v_9, v_{10}, v_{12}\}, \{v_7, v_{10}, v_{11}, v_{12}\}\}$\\
$A$ = $\{v_1, v_6\}$ and $B$ = $\{\{v_7, v_8, v_{10}, v_{12}\}, \{v_7, v_9, v_{10}, v_{11}\}, \{v_8, v_9, v_{11}, v_{12}\}\}$\\

Perform the Cartesian product of the corresponding $A$ and $B$, and add the hyperedges formed (along with their corresponding complementary hyperedges) into the hyperedge-set $E_{6}^{12}$. Note that, we have constructed a total of $30$ hyperedges in this step of our construction.

This completes the design of $S_{6}^{12}$. A total of $132$ hyperedges are present in its hyperedge-set
$E_{6}^{12}$, which are added through Step $(1)$ ($2$ hyperedges), Step $(2)$ ($60$ hyperedges) and Step $
(3)$ ($70$ hyperedges) of our construction. One can verify that the Steiner system we have devised here is isomorphic to the structure of $S(5,6,12)$ found in \cite{math} and \cite{curt}. Moreover, it is known that there exists a unique $S(5,6,12)$ Steiner system \cite{{col1}}.

\subsection{Constructing $S_{t}^{2t}$ when $t=4$}
We call the hypergraph corresponding to the Steiner system $S(3,4,8)$ as $S_{4}^{8}$. Similar to the previous section, we begin with an empty hyperedge-set of $S_{4}^{8}$, denoted by $E_{4}^{8}$, and the vertex-set of $S_{4}^{8}$ is denoted by $V_{4}^{8}$ = $\{v_1, v_2, v_3, \ldots, v_7, v_8\}$. Let us start devising the hyperedges for the hypergraph $S_{4}^{8}$:
\linebreak
\linebreak
\noindent $\textbf{Step (1)}$ Denote the hyperedge $\{v_1, v_2, v_3, v_4\}$ by $h$, and its
complementary hyperedge be $h'$. Note that $h'$ = $\{v_5, v_6, v_7, v_8\}$. We start the construction by adding the hyperedges $h$ and $h'$ into the hyperedge-set $E_{4}^{8}$.\\

\noindent $\textbf{Step (2)}$ Here we devise and add to $E_{4}^{8}$, all the hyperedges that intersect  exactly on $2$ vertices of $h$ containing a particular vertex, and $w.l.o.g.$ we choose $v_1$ to be that vertex. In addition, for every hyperedge added to $E_{4}^{8}$, we also add its corresponding complementary hyperedge. Now note the following sub-step:\\

$\textbf{Step (2.1)}$ If set $A$ = $\{1', x\}$ and set $B$ = $\{\{a, b\}, \{c, d\} \}$, then,\\

$\hspace{18mm}$ for set $A'$ = $\{1', y\}$, set $B'$ = $\{\{a, c\}, \{b, d\} \}$, and,

$\hspace{18mm}$ for set $A''$ = $\{1', z\}$, set $B''$ = $\{\{a, d\}, \{b, c\} \}$.\\ 

\noindent Note the following vertex-mapping:\\
$(1'\rightarrow v_1)$, $(x\rightarrow v_2)$, $(y\rightarrow v_3)$, $(z\rightarrow v_4)$, $(a\rightarrow v_5)$, $(b\rightarrow v_6)$, $(c\rightarrow v_7)$ and $(d\rightarrow v_8)$.\\\\
Apply the above step (Step $(2.1)$) on these mapped elements to get the following sets:\\
$A$ = $\{v_1, v_2\}$ and $B$ = $\{\{v_5, v_6\}, \{v_7, v_8\} \}$,\\
$A'$ = $\{v_1, v_3\}$ and $B'$ = $\{\{v_5, v_7\}, \{v_6, v_8\} \}$, and,\\
$A''$ = $\{v_1, v_4\}$ and $B''$ = $\{\{v_5, v_8\}, \{v_6, v_7\} \}$.\\

Perform the Cartesian products $A \times B$, $A' \times B'$ and $A'' \times B''$, and add the hyperedges thus formed (along with their corresponding complementary hyperedges) to $E_{4}^{8}$. 

This completes the design of $S_{4}^{8}$. A total of $14$ hyperedges are present in its hyperedge-set $E_{4}^{8}$, which are added in Step $(1)$ ($2$ hyperedges) and Step $(2)$ ($12$ hyperedges) of its construction.

\subsection{$2$-coloring properties of $S_{t}^{2t}$}

A hypergraph is said to be properly $2$-colorable if the vertices of the hypergraph can be colored using Red and Blue colors such that every hyperedge of the hypergraph contains Red as well as Blue colored vertices. The constructions defined makes it easier to study the $2$-coloring properties of the hypergraph $S_{t}^{2t}$ which is the Steiner system $S(t-1,t,2t)$.\\

\begin{lemma}
The hypergraphs $S_{6}^{12}$ and $S_{4}^{8}$ are properly $2$-colorable.
\end{lemma}
\begin{proof}
It is sufficient for us to show the existence of a proper $2$-coloring $\chi$ for the hypergraph $S_{t}^{2t}$. Let $h_1$ be a $t$-subset of $V_{t}^{2t}$ which is not a hyperedge in $S_{t}^{2t}$. One can easily observe that such a $t$-subset does exist as all the $t$-subsets of $V_{t}^{2t}$ are not in $S_{t}^{2t}$. Color all the vertices in $h_1$ by Red and the remaining vertices, $V_{t}^{2t} \setminus h_1$, by Blue and denote this coloring as $\chi$. Note that, for $S_{t}^{2t}$, the only way to get a monochromatic hyperedge in $\chi$ is by the presence of the hyperedge $h_1$ or $V_{t}^{2t} \setminus h_1$. However, $h_1$ is absent in $S_{t}^{2t}$, and $V_{t}^{2t} \setminus h_1$ is the complementary hyperedge corresponding to $h_1$. According to our construction (Step $(i)$ in the overview), the complementary hyperedge of $h_1$ will be present in the hypergraph $S_{t}^{2t}$ if and only if $h_1$ is also present in it. Hence both $h_1$ and $V_{t}^{2t} \setminus h_1$ are absent in the hypergraph $S_{t}^{2t}$. In other words, the $2$-coloring $\chi$ is a proper $2$-coloring of $S_{t}^{2t}$.\qed
\end{proof}


\small
\begin{thebibliography}{2200}

\bibitem{col1}
C.J. Colbourn and J. H. Dinitz.
Handbook of Combinatorial Designs, 2nd edition,
Chapman \& Hall/ CRC, Boca Raton (2007).

\bibitem{col2}
C.J. Colbourn and P.C. van Oorschot.
Applications of combinatorial designs in computer science,
ACM Comp. Surveys, 21 (1989): 223-250.

\bibitem{curt} 
R.T. Curtis,
The Steiner system $S(5,6,12)$, the Mathieu group $M_{12}$ and the ``kitten", in Atkinson, Michael D., Computational group theory (Durham), London: Academic Press (1984): 353-358.

\bibitem{math} 
E.L. Mathieu,
M\'{e}moire sur l'\'{e}tude des fonctions de plusieurs quantit\'{e}s, sur la mani\`{e}re de les former et sur les substitutions qui les laissent invariables,
Journal de Math\'{e}matiques Pures et Appliqu\'{e}es 6 (1861): 241-323.

\end{thebibliography}
\end{document}